\documentclass[11pt,a4wide]{article}

\usepackage[numeric]{amsrefs}
\usepackage{amsfonts, amsmath, amssymb, amsxtra, stmaryrd}
\usepackage{amsthm}

\usepackage[english]{babel}

\usepackage{enumerate}
\usepackage{comment}
\usepackage{mathrsfs}
\usepackage{graphicx}
\usepackage{caption,subcaption, float}
\usepackage{mathtools}
\usepackage{lscape}
\usepackage{rotating}
\usepackage{array}
\usepackage{tikz}

\usetikzlibrary{decorations.markings}

\usepackage{hyperref}

\numberwithin{subsection}{section}
\numberwithin{equation}{section}
\theoremstyle{plain}
\newtheorem{satz}{Theorem}[section]
\newtheorem{thm}[satz]{Theorem}
\newtheorem*{theorem*}{Theorem}
\newtheorem{lem}[satz]{Lemma}
\newtheorem{prop}[satz]{Proposition}

\newtheorem{cor}[satz]{Corollary}

\theoremstyle{definition}
\newtheorem{defi}[satz]{Definition}

\newtheorem{exam}[satz]{Example}

\newtheorem{setting}[satz]{Setting}
\newtheorem{remark}[satz]{Remark}

\allowdisplaybreaks

\newcommand{\eps}{\varepsilon}
\newcommand{\MG}{\mathfrak{g}}
\newcommand{\g}{\mathfrak{g}}
\newcommand{\ch}{\mathfrak{ch}}

\newcommand{\MF}{\mathbb{F}}
\newcommand{\MP}{\mathbb{P}}

\newcommand{\Pts}{\mathcal{P}}
\newcommand{\Lines}{\mathcal{L}}

\newcommand{\MSL}{\mathfrak{sl}_2}
\newcommand{\ME}{\mathcal{E}}
\newcommand{\E}{\mathcal{E}}

\newcommand{\MCP}{\mathcal{P}}
\newcommand{\MCF}{\mathcal{F}}
\newcommand{\ML}{\mathcal{L}}

\newcommand{\rad}{\mathop{\mathrm{rad}}}

\newcommand{\chf}{\overline{\ch_\Phi(\MF)}}

\begin{document}

\title{A geometric characterization of the classical Lie algebras }
\author{Hans Cuypers\footnote{corresponding author}, Yael Fleischmann \footnote{present address: 
Institut f\"ur Mathematik,
Warburger Str. 100,
33098 Paderborn,  yael.fleischmann@math.uni-paderborn.de
}\ \footnote{This work is (partly) financed by the Netherlands Organisation for
Scientific Research (NWO), project 613.000.905.}
\\
Department of Mathematics and Computer Science\\
Eindhoven University of Technology\\
P.O. Box 513, 5600 MB Eindhoven\\
The Netherlands\\
email: f.g.m.t.cuypers@tue.nl}

\maketitle

\section{Introduction}

An element $x\neq 0$ in a Lie algebra $\mathfrak{g}$ over a field $\mathbb{F}$  with Lie product $[\cdot,\cdot]$ is 
called  a {\em  extremal element} if  we have $$[x,[x,\mathfrak{g}]]\subseteq \mathbb{F}x.$$
(Or, in case the characteristic of $\mathbb{F}$ is two, when $[x,[x,\g]]=0$ and some additional conditions are satisfied. See Section \ref{sec:ee}.) 

Long root elements in classical Lie algebras are examples of 
pure extremal elements.
Arjeh Cohen \emph{et al.} \cite{CSUW01} initiated the investigation of
Lie algebras generated by extremal elements in order  
to provide a geometric characterization of the classical Lie algebras generated by their long root elements. 
He and Gabor Ivanyos studied the so-called {\em extremal geometry} 
with as points  the $1$-dimensional subspaces of $\MG$
generated by extremal elements of $\mathfrak{g}$ 
and as lines the  $2$-dimensional subspaces of $\g$ all whose nonzero vectors are extremal.
    
The main result of \cite{CI07} and \cite{CI06}, based on the work of Kasikova and Shult \cite{KS01}, states that, 
if  $\mathfrak{g}$ is of finite dimension, simple and generated by  extremal elements, then its extremal geometry either contains no lines 
or is the root shadow space of a spherical building. This raises the question whether one can recover the algebra $\mathfrak{g}$ from its extremal geometry.

In \cite{CRS14}, Cuypers, Roberts and Shpectorov show that in case the
building is of simply laced type, the Lie algebra 
$\mathfrak{g}$ is indeed a classical Lie algebra of the same type as the building and its extremal elements are the long root elements.

In this paper we prove the following result.

\begin{thm}\label{mainthm}
Let $\MG$ be a  simple Lie algebra generated by its set of  extremal elements.
If the extremal geometry of $\MG$ is the root shadow space of a spherical building of rank at least $3$, then $\MG$ is, up to isomorphism, uniquely determined by its extremal geometry.
\end{thm}
 
As we already noticed, the long root elements in (forms of) Chevalley Lie algebras are extremal (see Section \ref{sec:chevalley}).
So, as a corollary of the above result
we obtain for most of the simple classical Lie algebras that they  
are characterized by their extremal geometry.
The classes of simple classical Lie algebras, which are generated by their long root elements, but escape the 
above result, are those of relative type  $\mathrm{A}_1$, $\mathrm{B}_2$, $\mathrm{C_n}$ or $\mathrm{G}_2$.
The extremal geometries of these Lie algebras of type $\mathrm{A}_1$, $\mathrm{B}_2$ or $\mathrm{C}_n$ have extremal geometries with no lines.  

Simple Lie algebras $\mathfrak{g}$ over a field $\mathbb{F}$
generated by extremal elements whose extremal geometry has no lines 
have been studied in \cite{cuyfle}.
There it has been shown that either one can extend the underlying field
$\mathbb{F}$ quadratically and obtain an extremal geometry with lines,
or $\mathfrak{g}$ is a symplectic Lie algebra.
So, it only remains to study Lie algebras whose extremal geometry is a root shadow space of type $\mathrm{G}_2$.

We notice that finite dimensional simple Lie algebras defined over algebraically closed fields of characteristic $\geq 7$ and generated by  extremal elements have been shown to be classical by Premet \cite{Pre86-1}. In Theorem \ref{mainthm}, however,  no restriction is imposed on the field $\mathbb{F}$, nor on the dimension of the Lie algebra. Our results are independent of the classification of spherical buildings of rank $\geq 3$.

\bigskip

The present paper is organized as follows. 
In Section \ref{sec:rss} we provide definitions of, and some results on the geometric concepts used in this paper. In particular, we describe point-line geometries and  root shadow spaces,
and present some results on embeddings of point-line geometries in projective spaces. 

Section \ref{sec:ee} is devoted to  extremal elements. We present 
various results, mainly from \cite{CI07}, needed for the proof of our main results.
This includes the construction and classification of the extremal geometries
of finite dimensional simple Lie algebras generated by their extremal elements.

In Section \ref{sec:embedding} we start with the proof of Theorem \ref{mainthm}, by proving that
the embedding of the extremal geometry of a Lie algebra $\MG$ as in the hypothesis
of Theorem \ref{mainthm} into the projective space $\mathbb{P}(\MG)$ is determined by the isomorphism type of the geometry. 
This proof is continued in Section \ref{sec:unique}, where we show that, given
the isomorphism type of the extremal geometry of $\MG$ 
and its embedding into $\mathbb{P}(\MG)$, the Lie product
is fixed up  to a scalar.

In  Section \ref{sec:conclusions} we finish the proof of Theorem \ref{mainthm} by combining all the results obtained.

The final Section \ref{sec:chevalley} is devoted to examples.
In this section we show that the long root elements in  Chevalley Lie algebras
are extremal.

\bigskip
 
\noindent
{\bf Acknowledgment.} Parts of this paper can be found in the second author's PhD-thesis \cite{fle15}, which was written under supervision of Arjeh Cohen and the first author. 
We thank Arjeh Cohen, and also Kieran Roberts and Sergey Shpectorov, for many inspiring discussions on the topic. 

\section{Root shadow spaces and polarized embeddings}
\label{sec:rss}

In this section we discuss the geometric results needed to prove our main result,
Theorem \ref{mainthm}.

A {\em point-line geometry} $(\mathcal{P},\mathcal{L})$ consists of  a set $\mathcal{P}$ of {\em points} and a set $\mathcal{L}$ of {\em lines}, 
being subsets of the point set
$\mathcal{P}$ of size at least $2$.
A point-line geometry is called a {\em partial linear space} if any two points
are on at most one common line.

Let $(\mathcal{P},\mathcal{L})$ be a point-line geometry.
A subset $X$ of $\mathcal{P}$ is called a {\em subspace} of the point-line geometry    
$(\mathcal{P},\mathcal{L})$ if every line meeting $X$ in at least $2$
points is fully contained in $X$. We often identify a subspace $X$ with the point-line geometry $(X,\{l\in \mathcal{L}\mid l\subseteq X\})$.

The subset $X$ is called a {\em geometric hyperplane}, or just {\em hyperplane},
if it meets every line in one point, or contains the line.
Notice that a hyperplane is a subspace.
 
The collinearity graph of  $(\mathcal{P},\mathcal{L})$ is the graph with the points as vertices, and two points adjacent if and only if they are on a common line.

The distance between two points of $(\mathcal{P},\mathcal{L})$
is the distance in the collinearity graph, and 
the diameter of $(\mathcal{P},\mathcal{L})$ is the diameter of its collinearity graph.
If $p\in \mathcal{P}$ is a point and $d$ the diameter of $(\mathcal{P},\mathcal{L})$, then $p^\perp$ denotes the set of all points $q\in \mathcal{P}$ which are
at distance $<d$ from $p$.

\bigskip

A \emph{projective embedding} $\epsilon:\Gamma\rightarrow \mathbb{P}$
of a point-line geometry  $\Gamma=(\mathcal{P},\mathcal{L})$ into 
a projective space $\mathbb{P}$  is an injective map $\epsilon$ from $\MCP$ into the point set of the projective space $\MP$ such that 
the image of $\Pts$ spans $\MP$ and the image of any line in $\mathcal{L}$
comprises all projective points of a projective line in $\MP$. 
Note that this induces an injection from $\ML$ into the line set of $\MP(V)$. 
Moreover, if $\Gamma$ can be embedded into a projective space, then $\Gamma$ is a partial linear space.

Two embeddings $\epsilon:\Gamma\rightarrow \mathbb{P}$ and
$\epsilon':\Gamma\rightarrow \mathbb{P}'$ are called isomorphic, notation $\epsilon \cong \epsilon'$, if there is
an isomorphism $\phi:\mathbb{P}\rightarrow \mathbb{P}'$
with $\epsilon'=\phi\circ \epsilon$.

In this paper we only consider embeddings into projective spaces which are 
obtained from vector spaces over a field.
  
Let $V$ be a vector space  and  $\epsilon:\Gamma \rightarrow \MP(V)$ be  an 
embedding. Suppose $t:V\rightarrow W$ is a surjective 
semilinear transformation, with the property that  $K:=\text{ker}(t)$ intersects the span 
$\langle \epsilon(p),\epsilon(q)\rangle $ for any pair $p,q$ of  points in $\Pts$ trivially. 
Then $\epsilon$ can be carried onwards to the cosets of $K$, 
and we obtain an embedding $\epsilon':\Gamma \rightarrow \MP(W)$. 
We call $\epsilon'$ the \emph{morphic image} of $\epsilon$, or we say that $\epsilon'$ is \emph{derived} from $\epsilon$ or $\epsilon$ \emph{covers} $\epsilon'$. 
In particular, $\epsilon'(p):=t(\epsilon(p))$ is a $1$-space in $W$ for all $p\in \MCP$.
We write $\epsilon'=\epsilon/K$.

If all embeddings $\epsilon'$ of $\Gamma$ can be obtained as morphic images  from a fixed embedding $\epsilon$, 
then this $\epsilon$ is called \emph{absolute} or \emph{absolutely universal}.

Let $\epsilon:\Gamma\rightarrow \mathbb{P}(V)$ be an arbitrary projective embedding of $\Gamma$ into the projective space $\mathbb{P}(V)$ for some vector space $V$.
We call $\epsilon$ \emph{polarized} if and only if $\epsilon(p^\perp)$
is contained in a proper hyperplane of $\mathbb{P}$ for all $p\in\mathcal{P}$.

The {\em radical} $R_\epsilon$ of a polarized  embedding $\epsilon$ 
is the intersection 
$$R_\epsilon:=\bigcap_{p\in\mathcal{P}}\ \langle \epsilon(p^\perp)\rangle.$$

Here $\langle \epsilon(p^\perp) \rangle $ denotes the subspace of $V$ generated  by $\epsilon(p^\perp)$.

\begin{lem}\label{polarized}
Let $\psi:\Gamma\rightarrow \mathbb{P}$ be a projective embedding
covering  a polarized embedding $\phi$.
Then $\psi$ is polarized. 

Moreover, the kernel of the projection of $\psi$ to $\phi$ is contained in the radical of $\psi$.
\end{lem}

\begin{proof}
The first statement is trivial.

Now suppose, $\phi$ is an embedding of $\Gamma$ into the projective 
space $\mathbb{P}'$.
If the kernel $K$ of the projection $\tau$ of $\psi$  to $\phi$ is not contained in the radical of $\psi$, then there is an element $x\in \mathcal{P}$ such 
that $\langle \psi(x^\perp)\rangle$ does not contain $K$.
But that implies that the image under $\tau$ of the hyperplane $\langle \psi(x^\perp)\rangle$ of $\mathbb{P}$ is the full space $\mathbb{P}(W)$.
This contradicts that $\phi$ is polarized.
\end{proof}

\begin{prop}
Let $\psi$ be a cover of a polarized embedding $\phi$ of $\Gamma$.
If the radical of $\phi$ trivial, then $\phi$ is isomorphic to $\psi$ modulo its radical $R_\psi$.
\end{prop}

\begin{proof}
The projection $\tau$ of $\psi$ onto $\phi$ maps the radical of $\psi$ into the radical of $\phi$.
However, since the radical of $\phi$ is trivial, we find the kernel of $\tau$ to be the radical $R_\psi$.    
\end{proof}

\begin{satz}\label{polarizedembedding}
Suppose $\Gamma$ admits an absolutely universal embedding
and  a polarized embedding $\phi$ with trivial radical.

Then any polarized embedding $\psi$ of  $\Gamma$ covers $\phi$.
\end{satz}

\begin{proof}
Let $\epsilon$ be the absolutely universal embedding of $\Gamma$.
By Lemma \ref{polarized}, $\epsilon$ is polarized and both
$\psi$ and $\phi$ are isomorphic to the quotient of $\epsilon$ by a subspace
$K_\psi$ and $K_\phi$, respectively, of its radical $R_\epsilon$.

Since the radical of $\phi$ is trivial, we find $K_\phi$ to be equal to $R_\epsilon$.
But this implies that $K_\psi\subseteq K_\phi$ and $\psi$ clearly covers $\phi$.
\end{proof}

\begin{satz}\label{polarizedthm}
Suppose $(\mathcal{P, L})$ is a point-line geometry admitting an absolutely universal embedding  $\epsilon$. If $\phi$ is a polarized embedding of $(\mathcal{P, L})$, then 
\[ 
\phi/R_{\phi}\cong \psi/R_{\epsilon}.
\]

In particular, if $\phi$ has trivial radical, it is unique up to isomorphism.
\end{satz}
\begin{proof}
Lemma \ref{polarized} shows that $\phi\cong \epsilon/R$ for some $R\subseteq R_{\epsilon}$. The radical of $\epsilon/R$ is $R_{\epsilon}/R\cong R_{\phi}$, and we get
\[ \phi/R_{\phi}\cong (\epsilon/R)/R_{\phi} \cong (\epsilon/R)/(R_{\epsilon}/R)\cong \epsilon/R_{\epsilon}.\]
\end{proof}

We notice that similar results have  been obtained by Blok \cite{Blo11}.

\bigskip

The point-line geometries we want to consider in this paper are the so-called 
root shadow spaces of spherical buildings. 
We view a building as chamber system and use the notation of \cite{Wei03}. 

Let $\Delta$ be an irreducible spherical building of type $X_n$ and with corresponding root system $\Phi$. 
(In the case that the building is of type $\mathrm{BC}_n$, we only consider the type $\mathrm{B}_n$.) 
Fix the index set $I = \{1,\ldots,n\}$ of $\Delta$.

For any subset $J$ of $I$, we define a \emph{$J$-shadow} in $\Delta$ to be an $(I\setminus J)$-residue. 
If $j\in J$, then a \emph{$j$-line} is the set of all $J$-shadows containing chambers from a given $j$-panel. Let $\Pts$ be the set of $J$-shadows, 
called {\em points} and $\Lines$ be the set of $j$-lines for $j\in J$. We say that the point-line geometry $(\Pts, \Lines)$ is 
the \emph{$J$-shadow space} (or \emph{shadow space of type $X_{n,J}$}). 
If $J$ is the set $\{j\}$, then we write $X_{n,j}$ instead of $X_{n,J}$.

Of importance to us are the {\em root shadow spaces}. These are the shadow spaces of type $X_{n,J}$, (or $X_{n,j}$, when $J=\{j\}$)
where $J$ is a so-called {root set}, which can be defined as follows.
Fix a base $B = \{ \alpha_1, \ldots, \alpha_n \}$ for $\Phi$ and denote the longest root with respect to $B$ by $\alpha$. 
Then the {\rm root set} $J$ is the subset $I$ consisting of all elements $i\in I$ with $(\alpha,\alpha_i)\ne 0$. 
The root set is $J = \{1,n\}$ for $\mathrm{A}_n$, $J=\{2\}$ for $(\mathrm{BC})_n$, $\mathrm{D}_n$, $\mathrm{E}_6$ 
and $\mathrm{G}_2$, $J = \{ 1\}$ for $\mathrm{E}_7$ and $\mathrm{F}_4$, and $J = \{8\}$ for $\mathrm{E}_8$. 
(We use the Bourbaki labelling of the Dynkin diagrams as in \cite{Wei03}.)

\bigskip

We provide some more concrete descriptions of some examples 
of root shadow spaces.

\begin{exam}\label{exampleproj}
Let $\mathbb{P}$ be a projective geometry of projective dimension at least $2$
and suppose $\mathbb{H}$ is a set of hyperplanes of $\mathbb{P}$ forming a subspace of the dual of $\mathbb{P}$ 
such that the intersection of all hyperplanes in $\mathbb{H}$ is empty. 
The set $P$ is the set of incident
point-hyperplane pairs $(p,H)$ where $p\in \mathbb{P}$ and $H\in \mathbb{H}$. 
As lines in $L$ we take all  subsets of $P$  consisting 
of all $(p,H)\in P$ with  $p$ running over the points of a fixed line $\ell$ of $\mathbb{P}$ and $H$ a fixed hyperplane in $\mathbb{H}$ containing $\ell$,
or, dually, $p$ being a fixed point inside a fixed hyperline $K$ which is the intersection of two elements from $\mathbb{H}$ and $H$ running over all hyperplanes  containing $K$.  

If the dimension of $\mathbb{P}$ is $n<\infty$, then $(P,L)$ is a root shadow space of type $\mathrm{A}_{n,\{1,n\}}$.
\end{exam}

\begin{exam}\label{examplepolar}
Let $V$ be a vector space over a field $\mathbb{F}$ of characteristic $\neq 2$,
and $Q$ a nondegenerate quadratic form on $V$ with associated symmetric bilinear form $f$.
Assume that the Witt index of $Q$ is $n\geq 3$.

Let $P$ be the set of all isotropic $2$-spaces of $V$ and $L$ the set of all
subsets of $P$ consisting of all isotropic $2$-spaces on a fixed isotropic $1$-space and contained in a fixed isotropic $3$-space.
Then $(P,L)$ is a root shadow space of type  $\mathrm{BC}_{n,2}$ (or $\mathrm{D}_{n+1,2}$ in case $Q$ is split). 

More generally, if $\Pi$ is a nondegenerate polar space of rank $n$ at least $3$, then let $P$ be the set of all lines of $\Pi$ and $L$ the set of
all pencils of lines on a  point and inside a singular plane. Then, for finite $n$, the space $(P,L)$ is a root shadow space of type $\mathrm{BC}_{n,2}$ or $\mathrm{D}_{n+1,2}$. \end{exam}

We recall some well known properties of root shadow spaces and their collinearity graph.

\begin{prop}\cite{BC13}\label{rss}
Let $\Gamma=(\Pts,\Lines)$ be a root shadow space of a spherical building of rank at least $2$.
Then the following holds:
\begin{enumerate}[{\rm (a)}]
\item 
$\Gamma$ is a partial linear space.
\item
The diameter of the collinearity graph is $\leq 3$.
\item
Let $p$ be a point. The set $p^\perp$ is a hyperplane of $\Gamma$.
\end{enumerate}  
\end{prop}

We close this section with the following result of A. Kasikova and E. Shult and an application of Theorem \ref{polarizedthm}:

\begin{satz}[\cite{KS01}]\label{KasikovaShult}
Let $\Gamma=(\ME,\MCF)$ be an embeddable root shadow space of type $\mathrm{BC}_{n,2}$ $(n\geq 3)$, $\mathrm{D}_{n,2}$ $(n\geq 4)$, $\mathrm{E}_{6,2}$, $\mathrm{E}_{7,1}$, $\mathrm{E}_{8,8}$ or $\mathrm{F}_{4,1}$. Then $\Gamma$ admits an absolutely  universal embedding.
\end{satz}

\begin{proof}
 Kasikova and Shult prove the existence for each of these cases in \cite{KS01}. The case $\mathrm{BC}_{n,2}$ for $n\geq 4$ can be found in \cite{KS01}*{4.8} and $ \mathrm{BC}_{3,2}$ in 4.7, $ \mathrm{D}_{n,2}$ for $n\geq 5$ is covered in 4.5 and the special case $\mathrm{D}_{4,2}$ is treated in 4.1, $\mathrm{E}_{6,2}$, $\mathrm{E}_{7,1}$ and $\mathrm{E}_{8,8}$ in 4.11, and $\mathrm{F}_{4,1}$ in 4.9. 
\end{proof}

\begin{cor}\label{cor:KS}
Let $\Gamma=(\ME,\MCF)$ be a root shadow space of type $\mathrm{BC}_{n,2}$ $(n\geq 3)$, $\mathrm{D}_{n,2}$ $(n\geq 4)$, $\mathrm{E}_{6,2}$, $\mathrm{E}_{7,1}$, $\mathrm{E}_{8,8}$ or $\mathrm{F}_{4,1}$ admitting a polarized embedding. Then $\Gamma$ admits, up to isomorphism,  a unique polarized  embedding with trivial radical.
\end{cor}

\section{Extremal elements}
\label{sec:ee}

In this section we provide some definitions and collect some results on extremal elements, mainly from \cite{CI07}.

\begin{defi}
Let $\MG$ be a Lie algebra over the field $\MF$.
A nonzero element  $x\in \MG$ is called {\em extremal} if there is a map $g_x:\MG \rightarrow \MF$, called the {\em extremal form} on $x$, such that
\begin{equation} 
\label{extr}
\begin{aligned} 
\big[x,[x,y]\big]=2g_x(y)x
 \end{aligned} \end{equation}
and moreover
\begin{equation} \begin{aligned} \label{P1}
\big[[x,y],[x,z]\big]=g_x\big([y,z]\big)x+g_x(z)[x,y]-g_x(y)[x,z]
 \end{aligned} \end{equation}
and
\begin{equation} \begin{aligned} \label{P2}
\big[x,[y,[x,z]]\big]=g_x\big([y,z]\big)x-g_x(z)[x,y]-g_x(y)[x,z]
 \end{aligned} \end{equation}
for every $y,z\in \MG$.\\
The last two identities are called the {\em Premet identities}.
If the characteristic of $\mathbb{F}$ is not $2$, then the Premet identities follow from Equation \ref{extr}. See \cite{CI07}.
\end{defi}
As a consequence, $\big[ x,[x,\MG]\big]\subseteq \MF x$ for an extremal $x\in \MG$. 
We call  $x\in \MG$ a {\em sandwich} if $[x,[x,y]]=0$ and $[x,[y,[x,z]]]=0$ for every $y,z\in \MG$. So, a sandwich is an element $x$ for which the extremal form $g_x$ can be chosen to be identically zero. 
We introduce the convention that {\em $g_x$ is identically zero whenever $x$ is a sandwich in $\MG$}.
An extremal element is called {\em pure} if it is not a sandwich.

We denote the set of extremal elements of a Lie algebra by $E(\MG)$ or, if $\MG$   is clear from the context, by $E$.
Accordingly, we denote the set  $\{ \MF x| x\in E(\MG)\}$ of {\em extremal points} in the projective space on $\MG$ by $\mathcal{E}(\MG)$ or $\mathcal{E}$.

\begin{exam}\label{examplesl}
Let $V$ be a vector space over a field $\mathbb{F}$ with dual space $V^*$. Suppose $W^*$ is a subspace of $V^*$ annihilating $V$.

On $V\otimes W^*$ we can define a Lie bracket by linear extension of  the following product for pure tensors $v\otimes \phi$ and $w\otimes \psi$:
\begin{align*}
[v\otimes\phi, w\otimes \psi] =& (v\otimes \psi)\phi(w)-(w\otimes \phi)\psi(v).
\end{align*}

A pure tensor $v\otimes \phi$ is called singular if $\phi(v)=0$. 
Let $g$ be the $\mathbb{F}$-bilinear form on $V\otimes W^*$ defined by \[g(v\otimes \phi, w\otimes \psi)=-\psi(v)\phi(w)\] for $v\otimes \phi, w\otimes \psi \in V\otimes W^*$. 
Then for all singular pure tensors $v\otimes \phi$ and tensors $w\otimes\psi$   we have
\begin{align*}
\big[ v\otimes \phi,[v\otimes \phi, w\otimes \psi]\big]
=&\big[v\otimes \phi,\phi(w)v\otimes \psi-\psi(v)w\otimes \phi\big]\\
=&-\psi(v)\phi(w)v\otimes \phi
-\psi(v)\phi(w)v\otimes \phi\\
=& -2\psi(v)\phi(w)v\otimes \phi\\
=& 2 g(v\otimes \phi, w\otimes \psi)v\otimes \phi.
\end{align*}
In characteristic $\neq 2$ this implies that the singular pure tensors are extremal elements in  $V\otimes W^*$. This also holds true in characteristic $2$.
(It is straightforward, but somewhat tedious, to check that the Premet identities also hold.)

An element $v\otimes \phi\in V\otimes V^*$ acts linearly on $V$ by 
$$v\otimes\phi(w)=\phi(w)v$$
for all $w\in V$.
This provides an isomorphism between  $V\otimes V^*$ and the finitary general
 linear Lie algebra $\mathfrak{fgl}(V)$. (A linear map is finitary if its kernel has finite codimension.)
The singular pure tensors in $V\otimes V^*$ (or $V\otimes W^*$) generate (subalgebras of) $\mathfrak{fsl}(V)$, the finitary special Lie algebra, or $\mathfrak{sl}(V)$ in case $V$ has finite dimension.
\end{exam}

\begin{exam}\label{exampleorth}
Let $V$ be a vector space over the field $\mathbb{F}$ of characteristic $\neq 2$ and
$Q:V\rightarrow \mathbb{F}$ a nondegenerate quadratic form
with associated bilinear form $b$.
For $v\in V$ we denote by $b_v$ the linear form $w\in V\mapsto b(v,w)$.
 
Then, for $v,w\in V$ consider the following element in $V\otimes V^*$:
$$s_{v,w}=v\otimes b_w-w\otimes b_v.$$

If $v$ and $w$ are linearly independent and $Q(v)=Q(w)=b(v,w)=0$ 
then we call $s_{v,w}$ a {\em Siegel element}.
Now let $s_{v,w}$ be a Siegel element, then for all $x,y\in V$ we have 
\begin{align*}
b(s_{v,w}(x),y)&=b(b(w,x)v-b(v,x)w,y)\\
&=b(w,x)b(v,y)-b(v,x)b(w,y)\\
&=b(x,b(y,v)w-b(y,w)v)\\
&=-b(x,s_{v,w}(y)).
\end{align*}
The Siegel elements generate (and even linearly span) 
the orthogonal Lie algebra 
$\mathfrak{so}(V,b)=\{s\in \mathfrak{fsl}(V)\mid b(s(x),y)=-b(s(y),x)\}$.

A straightforward computation yields
\begin{align*}
[s_{v,w},[s_{v,w},s_{x,y}]]=(b(w,x)b(w,y)-b(v,x)b(v,y))s_{v,w},
\end{align*}
which shows that $s_{v,w}$ is extremal.
\end{exam}

\bigskip

For $x\in E$  and $\lambda\in\mathbb{F}$ we define the map
$\mathrm{exp}(x,\lambda):\MG\rightarrow \MG$ by

$$\mathrm{exp}(x,\lambda)y=y+\lambda[x,y]+\lambda^2g_x(y)x$$
for all $y\in \MG$.

\begin{prop}\cite[Lemma 15]{CI07}\label{exp}
Let $x\in E$ be a pure and $\lambda\in \mathbb{F}$. 
Then $\mathrm{exp}(x,\lambda)$ is an automorphism of $\MG$.
\end{prop}

Let $x\in E$ be pure.
By $\mathrm{Exp}(x)$ we denote the set $\{\mathrm{exp}(x,\lambda)\mid \lambda\in \mathbb{F}\}$.
Since, for $\lambda,\mu\in \mathbb{F}$ we have $\mathrm{exp}(x,\lambda)\mathrm{exp}(x,\mu)=\mathrm{exp}(x,\lambda+\mu)$,
 we find that $\mathrm{Exp}(x)$ is a subgroup of $\mathrm{Aut}(\MG)$ 
isomorphic to the additive group of $\mathbb{F}$.

Clearly, $\mathrm{Exp}(x)=\mathrm{Exp}(\lambda x)$ for $\lambda\in \mathbb{F}^*$. 
Therefore we can define  $\mathrm{Exp}(\langle x\rangle)$ to be equal to  $\mathrm{Exp}(x)$.

\begin{prop}\cite[Lemma 24-27]{CI07}
\label{2generators}
For $x,y\in E$ we have one of the following:

\begin{enumerate}[\rm (a)]
\item $\mathbb{F}x=\mathbb{F}y$;
\item $[x,y]=0$ and $\lambda x+\mu y\in E\cup \{0\}$ for all $\lambda,\mu\in \mathbb{F}$;
\item $[x,y]=0$ and $\lambda x+\mu y\in E$ only if $\lambda=0$ or $\mu=0$;
\item $z:=[x,y]\in E$, and $x,z$ and $y,z$ are as in case {\rm (b)};
\item $\langle x,y\rangle $ is isomorphic to $\mathfrak{sl}_2(\mathbb{F})$.
\end{enumerate}

Moreover, $g_x(y)\neq 0$ if and only if $\langle x,y\rangle $ is isomorphic to $\mathfrak{sl}_2(\mathbb{F})$.
\end{prop}

In each of the  cases described in the above proposition we find that the subalgebra generated by two extremal elements $x$ and $y$ is linearly spanned by all its extremal elements.
Indeed, this is trivial in the cases that $x$ and $y$ commute. If $x$ and $y$ do not commute, 
then $\mathrm{exp}(x,1)y=y+[x,y]+g_x(y)y$ is extremal and spans together with $x$ and $y$ the subalgebra $\langle x,y\rangle$.

Moreover, inside $\MSL$ we can see that $g_x(y)=g_y(x)$. This leads to the the following.

\begin{prop}\cite[Proposition 20]{CI07}\label{span}
\label{extremalform} 
Let $\MG$ be generated by $E(\MG)$. Then $\MG$ is linearly spanned by $E$ and 
there is a unique bilinear symmetric form $g:\MG \times \MG \rightarrow \MF$ such 
that the linear form $g_x$ coincides with $y\mapsto g(x,y)$ for each $x\in E$. 
Moreover, this form is associative in the sense that $g(x,[y,z])=g([x,y],z)$ for all $x,y,z\in \MG$.
\end{prop}

The form $g$ is called the {\em extremal form} on $\mathfrak{g}$.
As the form $g$ is associative, its radical $\rad(g):=\{ u\in \mathfrak{g} | g_u(z)=0 \ \forall\  z\in \mathfrak{g}\}$ is an ideal in $\mathfrak{g}$. 

Notice that the extremal form $f$ from \cite{CSUW01} satisfies $f=2g$.

\bigskip

\begin{defi}\label{namesofpairs}
For $x,y \in E$ extremal elements we define
\[ (x,y) \in \left\{
   \begin{array}{ll}
     E_{-2}, & \hbox{ $\Longleftrightarrow \mathbb{F}x=\mathbb{F}y$, } \\
    E_{-1}, & \hbox{ $\Longleftrightarrow [x,y]=0, (x,y)\notin E_{-2} \text{ and } \mathbb{F}x +\mathbb{F}y\subseteq E \cup \{0\}$ ,} \\
     E_{0}, & \hbox{ $\Longleftrightarrow [x,y]=0 \text{ and }  (x,y)\notin E_{-2}\cup  E_{-1} $,} \\
    E_{1}, & \hbox{ $\Longleftrightarrow [x,y]\neq 0 \text{ and } g(x,y)  =0 $,} \\
    E_{2}, & \hbox{ $\Longleftrightarrow g(x,y)  \neq 0 $.}
   \end{array}
 \right.
\]\\

For the corresponding extremal points $\langle x\rangle, \langle y\rangle$, we define
\[ (\langle x \rangle ,\langle y \rangle)\in \mathcal{E}_i \Longleftrightarrow (x,y) \in E_i.\]

Let $x\in E$. Then $y\in E_i(x)$ denotes that $(x,y)\in E_i$.  By $E_{\leq i}(x)$ we denote the set $\bigcup_{-2\leq j\leq i} E_j(x)$.
Similarly, if $x\in \mathcal{E}$, then $\mathcal{E}_i(x)$ consists of all
$y$ with  $(x,y)\in \mathcal{E}_i$, and   $\mathcal{E}_{\leq i}(x)$ denotes $\bigcup_{-2\leq j\leq i} \mathcal{E}_j(x)$.
\end{defi}

\begin{prop}\label{orbit}
If $x,y\in E$ with $g_x(y)\neq 0$, then $\langle x\rangle$ and $\langle y\rangle$ are in the same $\mathrm{Aut}(\g)$-orbit.
In particular, if the graph on $(\E,\E_2)$ is connected, then the set of all extremal points is  a single $\mathrm{Aut}(\g)$-orbit. 
\end{prop}

\begin{proof}
Let $x,y\in E$ with $g_x(y)=g_y(x)\neq 0$. Then, after replacing $y$ by a scalar multiple,  we can assume $g_x(y)=g_y(x)=1$.
But then $\mathrm{exp}(x,1)y=y+[x,y]+x=\mathrm{exp}(y,1)x$.
So, $\mathrm{exp}(y,-1)\mathrm{exp}(x,1)$ maps  $y$ to $x$. The proposition follows immediately. 
\end{proof}

\begin{defi}\label{defiextremalgeometry}
Let $\ME$ be the set of extremal points of the Lie algebra $\MG$ and let $\mathcal{F}$ be the set of projective lines $\MF x +\MF y$ for $(x,y)\in \ME_{-1}$. Hereby, we identify a $2$-space with the set of $1$-spaces it contains.
Then the point-line space $(\ME, \mathcal{F})$ is called the {\em extremal geometry} of $\MG$. We  denote it by $\Gamma(\MG)$, or in case it is clear what $\MG$ is, by $\Gamma$.

The {\em  rank} of $\Gamma(\MG)$ is the maximal dimension (as a linear subspace of $\MG$) of a 
subspace $X$ of $\Gamma(\MG)$ in which any two points are collinear.
\end{defi}

\begin{exam} 
Let $V$ be a vector space over a field 
$\mathbb{F}$ and $W^*$ a subspace of $V^*$ annihilating $V$.
Let $\ME$ be the set of extremal points of $\langle v\otimes \phi$
 where $v\in V$ and $\phi\in W^*$ with $\phi(v)=0$.
Then each extremal point $\langle v\otimes \phi$ corresponds to an incident point-hyperplane pair $(\langle v\rangle,\mathrm{ker}(\phi))$ of $\mathbb{P}(V)$.
The extremal geometry with point set $\ME$ is isomorphic to the geometry 
described in Example \ref{exampleproj}. In particular, if $V$ has dimension $n+1<\infty$, then the extremal geometry is a root shadow space of type $\mathrm{A}_{n,\{1,n\}}$.  

The extremal geometry is embedded into a subspace the projective  space
$\mathbb{P}(V\otimes W^*)$.
\end{exam}

\begin{exam}
Consider a vector space $V$ over a field $\mathbb{F}$ of characteristic $\neq 2$, and equipped with a nondegenerate quadratic form $Q$ of Witt index at least $2$.

Then we can consider the extremal geometry whose point set
is the set of extremal points generated by Siegel elements.
To each such extremal point  corresponds a unique isotropic line on $(V,Q)$.
The extremal geometry is isomorphic  to the geometry described in
Example \ref{examplepolar}.
Thus, if the Witt index of $Q$ is finite and greater than $2$, then 
the extremal geometry is a root shadow space of type $\mathrm{BC}_{n,2}$
or $\mathrm{D}_{n+1,2}$.

Again the extremal geometry is embedded into the subspace of
$\mathbb{P}(V\otimes V^*)$ generated by the extremal points.
\end{exam}

\begin{prop}\cite{CI07}\label{geometry}
Let $x,y\in \mathcal{E}$. 
Then we have the following:

\begin{enumerate}[{\rm (a)}]
\item $(x,y)\in \mathcal{E}_{-2}$ $\Leftrightarrow x=y$;
\item $(x,y)\in \mathcal{E}_{-1}$ $\Leftrightarrow$ $x$ and $y$ are distinct but collinear in $\Gamma(\MG)$;
\item $(x,y)\in \mathcal{E}_{0}$ $\Leftrightarrow$ $x$ and $y$ are at distance $2$ and have $\geq 2$ common neighbors in $\Gamma(\MG)$;
\item $(x,y)\in \mathcal{E}_{1}$ $\Leftrightarrow$ $x$ and $y$ are at distance $2$ and have a unique common neighbor in $\Gamma(\MG)$, the point $[x,y]$;
\item $(x,y)\in \mathcal{E}_{2}$ $\Leftrightarrow$ $x$ and $y$ are at distance $3$ in $\Gamma(\MG)$ $\Leftrightarrow$ $\langle x,y\rangle\cong \mathfrak{sl}_2(\mathbb{F})$.
\end{enumerate}
\end{prop}

The following lemma characterizes collinearity.

\begin{lem}\cite[Lemma 27]{CI07}\label{charcoll}
Let $x,y\in E$ be linearly independent. Then $(x,y)\in E_{-1}\Leftrightarrow$
there are $\lambda,\mu\in \mathbb{F}^*$ with $\lambda x+\mu y\in E$.
\end{lem}

We will use the following fundamental result of Cohen and Ivanyos (see \cite{CI06} and \cite{CI07}) in the next section.

\begin{satz}\cite[Theorem 28]{CI07} \label{CIthm1} 
Suppose that the extremal geometry $\Gamma$ of a Lie algebra $\MG$, generated by its set of  extremal elements and equipped with a nondegenerate extremal 
form $g$, has finite rank. Then a connected compontent of $\Gamma(\MG)$ 
is isomorphic to a root shadow space of type $\mathrm{A}_{n,\{1,n\}}$ $(n\geq 2)$, $ \mathrm{BC}_{n,2}$ $(n\geq 3)$, $ \mathrm{D}_{n,2}$ $(n\geq 4)$, $\mathrm{E}_{6,2}$, $\mathrm{E}_{7,1}$, $\mathrm{E}_{8,8}$, $\mathrm{F}_{4,1}$ or $\mathrm{G}_{2,2}$ or consists of a single point.

Furthermore $\MG$ is a direct sum of the Lie algebras generated by the connected components of  $\Gamma$.
Moreover, in each such Lie subalgebra the extremal points form a single orbit under the automorphism group of the Lie subalgebra.  
\end{satz}

\begin{proof}
The result directly follows from Theorem 28 of \cite{CI07}, except for the last sentence. 
But this follows from the fact that the subgraph $(\E,\E_2)$ induced on the connected components of $\Gamma$ is connected, cf. Lemma 5 of \cite{CI07}, and Proposition \ref{orbit}.
\end{proof}

Note that the labeling of the Coxeter diagrams follows \cite{Bou68}.

\section{The embedding}
\label{sec:embedding}

We fix the properties that we assume for Lie algebras in this section. 
If not mentioned otherwise, any Lie algebra in the remainder of this section is supposed to fulfill the conditions of Setting \ref{notation}.

\begin{setting}\label{notation}
By $\MG$ we denote a  Lie algebra generated by its
set $E$ of extremal elements and with nondegenerate extremal form $g$. 
By $\Gamma=(\ME, \mathcal{F})$, we denote the extremal geometry of $\MG$. We assume $\Gamma$ to be  connected and not a single point.
So, in particular we assume that $\ME_{-1}\neq \emptyset$. 
\end{setting}

Considering $\MG$ as a vector space, 
and its  projective geometry $\mathbb{P}(\MG)$, the {\em natural} projective embedding of the extremal geometry $\Gamma=(\ME, \mathcal{F})$ into $\mathbb{P}(\MG)$ is defined to be the injection
\[ \phi: \ME\rightarrow \mathbb{P}(\MG),\mathrm{\ with\ } \phi(x)=x\ \mathrm{\ for\ }x\in \ME.\]
By definition of lines in $\mathcal{F}$, 
the image under $\phi$ of a line $l\in \mathcal{F}$ is the full set $\phi(l)$ of points of some projective line in $\mathbb{P}(\MG)$. Moreover,
as the extremal points in $\ME$ linearly span $\MG$ (see ~\ref{span}), the set $\phi(\mathcal{E})$ spans $\mathbb{P}(\MG)$.
So, $\phi$ is indeed a projective embedding.

\begin{lem}
The embedding $\phi$ is polarized.
\end{lem}

\begin{proof}
For each  $x\in \mathcal{E}$ we find $\phi\big(\mathcal{E}_{\leq 1}(x)\big)$
to be contained in the hyperplane $\{y\in \MG\mid g(x,y)=0\}$. See Proposition \ref{geometry}.
\end{proof}

\begin{satz}\label{isoembeddings}
Suppose $\MG_1$ and $\MG_2$ are two Lie algebras as in Setting \ref{notation}, each of them 
generated by its set of extremal elements and equipped with a nondegenerate extremal form. 
Assume their corresponding extremal geometries $\Gamma_1$ and $\Gamma_2$ are isomorphic to 
each other and admit an absolutely universal embedding.
Then their natural embeddings are isomorphic.
\end{satz}

\begin{proof}
We can apply the results of Section ~\ref{sec:rss} and find
by \ref{polarizedthm} that the natural embeddings
$\phi_1$ and $\phi_2$ are isomorphic, provided their radicals are trivial.

Since, for $i=1,2$, the radical $R_i$ of  embedding $\phi_i$ is the intersection
of all the subspaces $\langle \mathcal{E}_{\leq 1}(x)\rangle$ where $x$ runs through 
the set of extremal points of $\MG_i$, we find these radicals to be contained
in the radical of the extremal form $g_i$ of $\MG_i$.

As the radicals of the forms $g_1$ and $g_2$ are trivial by assumptions, 
the radicals of the embeddings are also trivial.
\end{proof}

The above has the following consequences.

\begin{cor}\label{cor524}
Let $\MG_1$and $\MG_2$ be Lie algebras as in \ref{notation}.
Assume the corresponding extremal geometries $\Gamma_1$ and $\Gamma_2$
are isomorphic to each other and to a connected root shadow space of type 
$\mathrm{BC}_{n,2}$ $(n\geq 3)$, $ \mathrm{D}_{n,2}$ $(n\geq 4)$, $\mathrm{E}_{6,2}$, $\mathrm{E}_{7,1}$, $\mathrm{E}_{8,8}$, or $\mathrm{F}_{4,1}$. 
Then their natural embeddings are isomorphic.
\end{cor}

\begin{proof}
By Corollary \ref{cor:KS} we find that $\Gamma_i$, with $i=1,2$, 
admits an absolutely universal embedding. So Theorem \ref{isoembeddings} applies.
\end{proof}

\begin{remark}
For root shadow spaces of type $\mathrm{A}_{n,\{1,n\}}$ and of type $\mathrm{G}_{2,2}$
it is not known whether they admit an absolutely universal embedding.

The results of V\"olklein \cite{Voe89} imply that the natural embeddings of the
extremal geometries of type $\mathrm{A}_{n,\{1,n\}}$ and of type $\mathrm{G}_{2,2}$
of the  Chevalley Lie algebras of type $\mathrm{A}_n$ and $\mathrm{G}_2$ 
do have a universal cover (see Section \ref{sec:chevalley}).

Blok and Pasini  \cite{BP03} obtain some partial results on embeddings
of the geometries of type $\mathrm{A}_{n,\{1,n\}}$ under some extra conditions on the underlying field. 
Van Maldeghem and Thas \cite{TVM04} show that the natural embedding of finite dual Cayley hexagons in the Chevalley Lie algebra is, 
up to isomorphism, the unique embedding of the hexagon in dimension $\geq 14$.
\end{remark}

In the next section we will prove that given an embedding of the extremal geometry of a Lie algebra there is, up to a scalar multiple, 
at most one Lie bracket corresponding to it.
This implies that the Lie structures $\MG_1$ and $\MG_2$, in the cases considered in \ref{cor524} are isomorphic.

\section{Uniqueness of the Lie product}
\label{sec:unique}

Let $\MG$ be a Lie algebra generated by its set of extremal elements $E$
with respect to a nondegenerate extremal form $g$.
As before let $\Gamma=(\ME,\MCF)$ be the extremal geometry of $\MG$. 
In the previous section we have seen  that the natural embedding of 
$\Gamma$ into the projective space $\mathbb{P}(\MG)$
is uniquely determined (up to isomorphism), 
if $\Gamma$ admits an absolutely  universal embedding. 
Our goal is to prove that not only the embedding of the extremal geometry
is uniquely determined, but that also the Lie product is determined up to scalar
multiples.

In this section we assume the following.

\begin{setting}\label{setting2}
Let $\MG$ be a Lie algebra as in \ref{notation}, with $\Gamma$ naturally embedded into the projective space $\MP(\MG)$. \\
Let $[\cdot,\cdot]$ denote the Lie product on $\MG$.
We consider  a second Lie product $[\cdot,\cdot]_1$ defining a Lie algebra $\MG_1$ 
on the vector space underlying $\MG$ with  extremal form $g_1$ and also $\Gamma$ as extremal geometry.
\end{setting}

We want to show that $[\cdot,\cdot]_1=\lambda[\cdot,\cdot]$ for some fixed $\lambda\in \MF^*$.
Notice that the relations $\ME_i$ with $-2\leq i\leq 2$ are determined by $\Gamma$ (see \ref{geometry}). 
So elements $x,y\in \ME$ are in relation $\ME_i$ in $\MG$ if and only if they are in relation $\ME_i$ in $\MG_1$.

\begin{lem}\label{scalarforE1}
Let $(x,y)\in E_{\leq 1}$, then there is a $\lambda\in \MF^*$
such that $[x,y]_1=\lambda [x,y]$.
\end{lem}

\begin{proof}
If  $(x,y)\in E_{\leq 0}$, then $[x,y]_1=0=[x,y]$.

If $(x,y)\in E_{1}$, then, by Proposition \ref{geometry},  both $[x,y]$ and $[x,y]_1$ span the unique point
in $\ME$ collinear to both $\langle x\rangle$ and  $\langle y\rangle$.
So indeed, there is a $\lambda\in \MF^*$ with $[x,y]_1=\lambda [x,y]$.
\end{proof}

Now we concentrate on the subalgebra of $\MG$ generated by 
a pair of points in $\mathcal{E}_2$. 
Such a subalgebra is isomorphic to $\MSL(\MF)$.

\begin{lem}\label{sl2invariant}
Let $(x,y)\in E_{2}$ be a generating a subalgebra $\mathfrak{h}$ of $\MG$. 
Then  there is a $\lambda\in \MF^*$ such that for all
$v,w\in \mathfrak{h}$ we have $[v,w]_1=\lambda[v,w]$.
\end{lem}

\begin{proof}
Without loss of generality suppose that $g(x,y)=1$.
Inside both $\MG$ and   $\MG_1$ the elements $x$ and $y$ generate a subalgebra $\mathfrak{h}$ and $\mathfrak{h}_1$, respectively, 
isomorphic to $\mathfrak{sl}_2$.

We first prove that $\mathfrak{h}$ and $\mathfrak{h}_1$
are equal as linear subspaces of $\mathfrak{g}$.

Inside  $\Gamma(\MG)$ we take two distinct lines $l_1$ and $l_2$
on $\langle x\rangle$ with $[l_1,l_2]=\langle x\rangle$. Notice that such lines exist.
For $i=1,2$, fix a point $\langle x_i\rangle$ on $l_i$ which is at distance $2$
from $\langle y\rangle$.
Let $y_i:=[y,x_i]$.
Then for each point $\langle z_1\rangle$ on the line through $\langle x_1\rangle$ and  $\langle y_1\rangle$, there is a  point $\langle z_2\rangle$ on the line through $\langle x_2\rangle$ and  $\langle y_2\rangle$ which is in relation $\ME_1$ with $\langle z_1\rangle$.  This follows from the observation that the group $\langle \mathrm{Exp}(x),\mathrm{Exp}(y)\rangle$ leaves the lines $\langle x_1,y_1\rangle$ and $\langle x_2, y_2\rangle$ invariant and is transitive on the points of these lines.

We claim that both in $\MG_1$ and $\MG_2$
the unique common neighbour $\langle z\rangle=\langle [z_1,z_2]\rangle$ 
of $\langle z_1\rangle$ and $\langle z_2\rangle$ is inside the subalgebra
generated by $x$ and $y$.  

Indeed, within $\mathrm{Aut}(\MG)$ we find that the elements of $\mathrm{Exp}(\langle x\rangle)$ fix, 
for $i=1,2$, the point $\langle x_i\rangle$
as well as the line spanned by $x_i$ and $y_i$. Moreover,
$\mathrm{Exp}(\langle x\rangle)$ acts transitively on the points of 
this line different from $\langle x_i\rangle$. 
Thus there is an element $e\in\mathrm{Exp}(\langle x\rangle)$ that maps
$\langle y_1\rangle$ to $\langle z_1\rangle$.
As $e$ leaves the line spanned by $x_2$ and $y_2$ invariant,
it maps $\langle y_2\rangle$ to the unique point on this line which is at distance $2$ 
from $\langle z_1\rangle$, the point $\langle z_2\rangle$.
But then $\langle y\rangle$ is mapped to $\langle z\rangle$ by the element $e$, 
as $\langle z\rangle$ is the unique common neighbor of $\langle z_1\rangle$ and $\langle z_2\rangle$.
This clearly implies that $\langle z\rangle$ is inside the subalgebra $\mathfrak{h}$ of $\MG$
generated by $x$ and $y$.  
In particular, $x,y$ and $z$ linearly span the subalgebra $\mathfrak{h}$ of $\MG$
generated by $x$ and $y$.

But similarly, these three elements are also  contained in  the subalgebra
$\mathfrak{h}_1$
of $\MG_1$ generated by $x$ and $y$ and span this subalgebra.
So, the subalgebras $\mathfrak{h}$ and $\mathfrak{h}_1$ have to coincide as linear subspaces. 
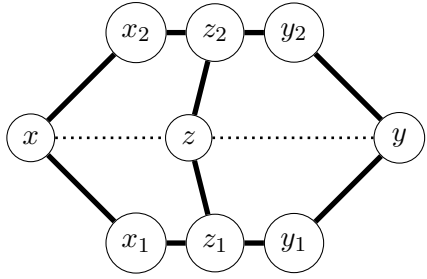
\begin{figure}
\begin{center}
\begin{tikzpicture}[scale=0.7]

\tikzstyle{every node}=[circle,draw];
\draw (0,0) node (a) {$x$};
\draw (2,-2) node(b) {$x_1$};
\draw (2,2) node (c) {$x_2$};
\draw (5,-2) node(d) {$y_1$};
\draw (5,2) node(e) {$y_2$};
\draw (7,0) node (f) {$y$};
\draw (3.5,-2) node(g) {$z_1$};
\draw (3.5,2) node(h) {$z_2$};
\draw (3,0) node(i) {$z$};

\draw[-, line width=2pt] (a)--(b);
\draw[-, line width=2pt] (a)--(c);
\draw[-, line width=2pt] (b)--(g);
\draw[-, line width=2pt] (g)--(d);

\draw[-, line width=2pt] (c)--(h) ;
\draw[-, line width=2pt] (h)--(e);

\draw[-, line width=2pt] (f)--(d);
\draw[-, line width=2pt] (f)--(e) ;

\draw[-, line width=2pt] (g)--(i);
\draw[-, line width=2pt] (h)--(i) ;

\draw[dotted, line width=1pt] (a)--(i) ;
\draw[dotted, line width=1pt] (i)--(f) ;

\end{tikzpicture}

\caption{Configuration of points}
\end{center}
\end{figure}
The above proves more.
Indeed, it shows that
the point $\langle z\rangle$ is in the $\mathrm{Exp}(\langle x\rangle)$-orbit
of $\langle y\rangle$ both with respect to $[\cdot,\cdot]$ and to $[\cdot,\cdot]_1$. Actually, these two orbits have to be equal.

In $\MG$ this orbit, together with $\langle x\rangle$, consists of all 
$1$-spaces spanned by elements 
\[ a x+by+c[x,y],
\]
where $a,b,c\in \mathbb{F}$ satisfy $ab=c^2$.

Now, suppose
\[ [x,y]_1 = \alpha x +\beta y +\gamma [x,y], \]
for some fixed $\alpha, \beta, \gamma \in \mathbb{F}$. 
Note that $[x,y]_1 \neq 0 \neq [x,y]$, since $(x,y)\in E_2$  
with respect to both Lie products by the general assumptions in \ref{setting2}.
Note moreover that $\gamma\neq 0$, 
since otherwise  \ref{charcoll} leads to a contradiction.

The images of $ y$
under elements from  $\mathrm{Exp}(\langle x\rangle)$, but now with respect to $[\cdot,\cdot]_1$,
are of the form
\[ y+\lambda [x,y]_1+\lambda^2 g_1(x,y) x=
(\lambda^2g_1(x,y) + \lambda \alpha)x+( 1+\lambda\beta)y+\lambda^2\gamma^2[x,y]
\] where $\lambda\in \MF$.
These elements are also extremal in $\MG$ and hence satisfy the equation
\[ 
 ( 1+\lambda\beta)(\lambda^2g_1(x,y) + \lambda \alpha) = \lambda^2\gamma^2.
\]
This implies that the qubic equation
\[( 1+\beta X)(g_1(x,y)X^2 + \alpha X) =  \gamma^2 X^2
\]
has  $|\MF|$ zeros.
So, if $|\MF|> 3$, this means
\begin{align*} 
  \alpha=\beta=0,\ \mathrm{and}\ \gamma^2 = g_1(x,y),
\end{align*}
and we deduce
\begin{align*}
[x,y]_1 = \gamma [x,y].
\end{align*}
But this implies that $[\cdot,\cdot ]_1$  equals $\gamma [\cdot ,\cdot ]$ and hence is a scalar multiple of $[ \cdot , \cdot ]$.

In case that $|\MF|=2$, the above equation for $\lambda=1$ reads as follwos:
\[ (1+\alpha)(1+\beta)=\gamma^2.\]
Now if $\alpha=1$ or $\beta=1$, it follows that $\gamma^2=0$ and so $\gamma=0$, which is a contradiction. So also here, $\alpha=\beta=0$ must hold.\\
In case that $|\MF|=3$, then setting $\lambda$ to be $\pm 1$ 
implies that we have the following two equations for $\alpha$ and $\beta$:
\begin{align*}
(1+\beta)(g_1(x,y)+\alpha)=1\\
(1-\beta)(g_1(x,y)-\alpha)=1.
\end{align*}
But then $\alpha=\beta=0$ and $g_1(x,y)=1$.
So, we conclude that the Lie product is unique up to scalar multiples.
\end{proof}

\begin{lem}\label{constantonline}
Let $x\in E$ and $l\in \mathcal{F}$. Suppose $y_1,y_2\in E$ span $l$. 
Then we can find an element $\lambda\in \MF^*$ with $[x,y_i]_1=\lambda[x,y_i]$ for both $i=1$ and $i=2$. 
\end{lem}

\begin{proof}
Under the given conditions, Lemma \ref{scalarforE1} and Lemma \ref{sl2invariant} imply that there exist $\lambda_i$ for $i=1,2$  with $[x,y_i]_1=\lambda_i[x,y_i]$.
If $[x,y_1]=0$ or $[x,y_2]=0$, then clearly we can take $\lambda_1$ and $\lambda_2$ to be equal. 
So assume $[x,y_1]\neq 0\neq [x,y_2]$ and let $y_3:=-(y_1+y_2)$ such that there exists $\lambda_3$ with $[x,y_3]_1=\lambda_3[x,y_3]$.

Suppose $\lambda_1\neq \lambda_2$.
We find  
\begin{align*}0=[x,y_1+y_2+y_3]=[x,y_1]+[x,y_2]+[x,y_3]\end{align*}
and \begin{align*}0=&[x,y_1+y_2+y_3]_1=[x,y_1]_1+[x,y_2]_1+[x,y_3]_1\\
=&\lambda_1[x,y_1]+\lambda_2[x,y_2]+\lambda_3[x,y_3].\end{align*}
But this implies that $(\lambda_1-\lambda_2)[x,y_1]+(\lambda_3-\lambda_2)[x,y_3]=0$
and hence 
\begin{align*}
0=& [x,(\lambda_1-\lambda_2)y_1+(\lambda_3-\lambda_2)y_3]\\
=&[x,(\lambda_1-\lambda_2)y_1+(\lambda_3-\lambda_2)(-y_1-y_2)]\\
=&[x,(\lambda_1-\lambda_3)y_1+(\lambda_2-\lambda_3)y_2].
\end{align*}

With the definition $z:=(\lambda_1-\lambda_3)y_1+(\lambda_2-\lambda_3)y_2\neq 0$, we have $[x,z]=0$  and we find $z\in \ME_{\leq 0}(x)$ and hence $[x,z]_1=0$ by Lemma \ref{scalarforE1}.
Since $[x,y_1]\neq 0\neq [x,y_2]$, the element $z$ is not a multiple of $y_1$ or $y_2$, and hence
there are $\mu_1,\mu_2\in \MF^*$ with $y_2=\mu_1y_1+\mu_2 z$.
Now we find 
\begin{align*}[x,y_2]_1=&[x,\mu_1y_1+\mu_2 z]_1
=[x,\mu_1 y_1]_1+[x,\mu_2 z]_1\\
=&\lambda_1\mu_1[x,y_1]+0
=\lambda_1\mu_1 [x,y_1]
\end{align*}
and
\begin{align*}[x,y_2]=[x,\mu_1y_1+\mu_2 z]=\mu_1[x,y_1]+0=\mu_1[x,y_1].\end{align*}
So, $[x,y_2]_1=\lambda_1[x,y_2]$, contradicting $\lambda_1$ to be different from $\lambda_2$.
\end{proof}

We are now ready to prove the main objective of this section.

\begin{satz}\label{uniquenesslieproduct}
Let $\MG$ be a Lie algebra generated by its set of extremal elements $E$, equipped with the Lie product denoted by $[\cdot, \cdot]$ and a nondegenerate extremal form $g$. 

Assume that there is a second Lie product $[\cdot, \cdot]_1$ defined on the underlying vector space, with corresponding nondegenerate extremal form $g_1$, giving rise to the same extremal geometry $\Gamma$. 
Then, there is a $\lambda\in \MF^*$ with $[x,y]_1=\lambda [x,y]$ for all $x,y\in \MG$.
\end{satz}

\begin{proof}
Fix a pair $(x,y)\in E_2$. Then, by Lemma \ref{sl2invariant},  there is a $\lambda\in \MF^*$ with $[x,y]_1=\lambda [x,y]$.
We prove that this element $\lambda$ is the one we are looking for.
We begin with the proof that for all $z\in E$ we have $[x,z]_1=\lambda[x,z]$.

Suppose $z\in E$ is different from $y$. 
If $[x,z]=0$, then $z\in E_{\leq 0}(y)$, hence also $[x,z]_1=0$, and $[x,z]_1=\lambda[x,z]$.

If $z\in E_2(x)$, then, as  $\Gamma$ has diameter $3$, we can find elements
$z_1$ and $z_2$ in $E_{\geq 1}(x)$  such that $\langle y,z_1\rangle$, $\langle z_1,z_2\rangle$, and $\langle z_2,z\rangle$ are in $\mathcal{F}$.
(Notice that we allow these subspaces to be equal to each other.) 
Now we can apply the above Lemma \ref{constantonline} to each of these lines and eventually find that as $[x,z_1]\neq 0\neq [x,z_2]$ that $[x,z]_1=\lambda [x,z]$.

Finally consider the case where $z\in E_1(x)$. 
Notice that $[x,z]\in E$.
As $g$ is nondegenerate, we can find an element $u\in E_2([x,z])$.
This $u$ is in $E_1(z'')$ for some $z''\in\langle [x,z],z\rangle\setminus \langle [x,z]\rangle$.
As $\mathrm{Exp}(\langle x\rangle)$ is transitive on the points of $\langle [x,z],z\rangle$ distinct from $\langle [x,z]\rangle$, we can assume that $u\in E_1(z)$. Now  take $z'$ to be $[u,z]$. Then $z'\in E_1([x,z])$.
If $z'\in E_{\leq 0}(x)$, then $[[x,z],z']=-[[z,z'],x]-[[z',x],z]=0$, which is a contradiction with $z'\in E_1([x,z])$.
So, $z'\in E_1(x)$, but then, as $z+z'\in E$, we have  $[x,z+z']=[x,z]+[x,z']\in E$ and, by \ref{charcoll}, $z'\in E_{-1}([x,z])$. Again a contradiction.
Hence, $z'\in E_2(x)$.

By the above we have $[x,z']_1=\lambda [x,z']$ and Lemma \ref{constantonline} implies now that $[x,z]_1=\lambda [x,z]$.

Since we started with a fixed $x\in E$, it remains to show that the scalar factor $\lambda$ is independent of $x$.
Clearly, multiplying $x$  with a nonzero scalar, does not change the value of $\lambda$.

Suppose $x_1\in E$ with $(x,x_1)\in E_{-1}$.
Then let $z\in E_2(x)$. Then 
$[z,x]_1=\lambda [z,x]$. If  $z\in E_2(x_1)$ then applying the above to $z$ instead of $x$ we find $[z, x_1]_1=\lambda [z, x_1]$. So also $[x_1,z]_1=\lambda [x_1,z]$. If $z\not \in E_2(x_1)$, then let $x_2=x+x_1$.
The element $x_2$ is extremal and $z\in E_2(x_2)$, so, by similar  arguments as used  above, we have $[x_2,z]_1=\lambda [x_2,z]$.
Now pick an element $z_1\in E_2(x_1)$. Then  $z_1\in E_2(x)$ and $[x,z_1]_1=\lambda [x,z_1]$, or $z_1\in E_2(x_2)$  and $[x_2,z_1]_1=\lambda [x_2,z_1]$. 
As above we get  $[x_1,z_1]_1=\lambda [x_1,z_1]$ and $\lambda$ is also the scalar for $x_1$.

Now connectedness of $\Gamma(\MG)$ implies that  the scalar is the same for all extremal elements in $E$.

But then $\lambda$ is the same for all  pairs $(x,y)$ of extremal elements in $E\times E$. Since $E$ generates $\MG$, we find for all $x,y\in\MG$ that $[x,y]_1=\lambda[x,y]$.
\end{proof}

\section{Conclusions}\label{sec:conclusions}

Combining the results of the previous two sections, we finally can characterize the Lie algebras under consideration by their extremal geometry. 

\begin{satz}\label{MainResult5}
Let $\MG$ be a Lie algebra generated by its set $E$
of extremal elements with respect to the extremal form $g$ with trivial radical. If  the natural embedding of the extremal geometry $\Gamma(\MG)$ into $\MP(\MG)$ admits an absolutely universal cover, 
then $\MG$ is uniquely determined (up to isomorphism) by $\Gamma(\MG)$.
\end{satz}

\begin{proof}
We combine the previous results. So let $\MG_1$ be a second Lie algebra with isomorphic extremal geometry $\Gamma(\MG_1)\cong \Gamma(\MG)$. By ~\ref{isoembeddings}, the projective embeddings of $\Gamma(\MG)$ and $\Gamma(\MG_1)$ are isomorphic 
and therefore $\g$ and $\g_1$ are not only isomorphic vector spaces, but
the they also have the same Lie structure, as a consequence of Theorem  ~\ref{uniquenesslieproduct}. 
\end{proof}

\begin{cor}\label{thmfortypes}
Let $\MG_1,\MG_2$ be two Lie algebras as in \ref{notation} with extremal geometries $\Gamma(\MG_1)\cong\Gamma(\MG_2)$ isomorphic to a root shadow space of type $\mathrm{BC}_{n,2}$ $(n\geq 3)$, $ \mathrm{D}_{n,2}$ $(n\geq 4)$, $\mathrm{E}_{6,2}$, $\mathrm{E}_{7,1}$, $\mathrm{E}_{8,8}$, or $\mathrm{F}_{4,1}$. Then $\MG_1\cong \MG_2$.
\end{cor}

\begin{proof}
As a consequence of the main result \ref{KasikovaShult} in \cite{KS01},  the root shadow space of $\MG_1$ (and $\MG_2$) of one of the given types has an absolutely universal cover. So Theorem \ref{MainResult5} applies.
\end{proof}

The above results provide a proof of Theorem \ref{mainthm}.
Indeed, suppose $\g$ is a simple Lie algebra generated by its set $E$ of extremal elements. Then the extremal form $g$ is nondegenerate. By the assumptions of Theorem \ref{mainthm} and 
Theorem \ref{CIthm1} we find the extremal geometry $\Gamma(\g)$ to be a root shadow space of type $\mathrm{A}_{n,\{1,n\}}$ $(n\geq 3)$, $ \mathrm{BC}_{n,2}$ $(n\geq 3)$, $ \mathrm{D}_{n,2}$ $(n\geq 4)$, $\mathrm{E}_{6,2}$, $\mathrm{E}_{7,1}$, $\mathrm{E}_{8,8}$, or $\mathrm{F}_{4,1}$.
Now we can apply the above corollary to find that $\g$ is uniquely determined by its extremal geometry, unless the extremal geometry is of type $\mathrm{A}_{n,\{1,n\}}$ $(n\geq 3)$. In that case we refer to \cite{CRS14}.

\section{Extremal elements in classical Lie algebras}

\label{sec:chevalley}

In this final section we investigate extremal elements in 
classical Lie algebras. 
Although most of the results are well known, we present them now within the theory of extremal elements.

Let $\Phi$ be an irreducible root system with dot product $( \cdot,\cdot )$ and $\Phi^+$ the set of positive roots.
Then consider $$\ch_{\Phi}=\bigoplus_{\alpha\in \Phi^+}\mathbb{Z}x_\alpha+\mathbb{Z}x_{-\alpha}+\mathbb{Z}h_\alpha$$
where for $\alpha\in \Phi^+$ the elements $x_\alpha$, $x_{-\alpha}$ and $h_\alpha$ form a basis. Let $h_{-\alpha}=-h_{-\alpha}$.
Then on $\ch_{\Phi}$ we define a bilinear product 
$$[\cdot,\cdot]:\ch_{\Phi}\times \ch_{\Phi}\rightarrow \ch_{\Phi}$$
by the following rules:
\begin{align*}
[h_\alpha,h_\beta]=&0\\
[h_\alpha,x_{\beta}]=&2\frac{( \beta, \alpha)}{( \alpha,\alpha)} x_{\beta}\\
[x_{\alpha},x_{\beta}]=&\begin{dcases} N_{\alpha, \beta}x_{\alpha+\beta} & \text{if } \alpha+\beta\in\Phi , \\
h_\alpha & \text{if } \beta=-\alpha,\\ 0 &\text{otherwise}, \end{dcases}
\end{align*}
 where $\alpha, \beta\in \Phi$.  The numbers $N_{\alpha, \beta}$ are integral \textit{structure constants} chosen to be $\pm (p_{\alpha, \beta}+1)$, where $p_{\alpha, \beta}$ is the biggest number such that $-p_{\alpha, \beta}\alpha+\beta$ is a root.

For a suitable choice of signs this product defines a Lie algebra called the  \textit{integral Chevalley Lie algebra}, see \cite{Tits66,Car72}. The formal basis elements $x_{\alpha}$, $ \alpha\in \Phi$, and $h_\alpha$ with $\alpha$ positive,  form a \textit{Chevalley basis} of $\ch_{\Phi}$. 
A Lie algebra $\ch_\Phi(\MF):=\ch_{\Phi}\otimes \MF$ obtained by tensoring with a field $\MF$ will be called a \textit{Chevalley Lie algebra}.

In this section, we investigate extremal elements in a Chevalley Lie algebra
$\ch_\Phi(\MF)$ over a field $\mathbb{F}$.
To avoid several technical difficulties, we restrict our attention to the case
that the characteristic of $\mathbb{F}$ is different from $2$. (Notice that most results remain true in the characteristic $2$ case.)

\begin{exam}
Let $V$ be a vector space over the field $\mathbb{F}$ of dimension $n+1$.
If $e_1,\dots,e_{n+1}$ is a basis of $V$ with dual basis $\phi_1,\dots,\phi_{n+1}$, then the elements $x_{\epsilon_i-\epsilon_j}=e_i\otimes \phi_j$, with $i\neq j$
and $h_{\epsilon_i-\epsilon_j}=e_i\otimes\phi_i-e_j\otimes\phi_j$ form a Chevalley basis of type $\mathrm{A}_n=\{\epsilon_i-\epsilon_j\mid 1\leq i\neq j\leq n+1\}$ for the  Lie algebra $\mathfrak{sl}(V)$ of Example \ref{examplesl}.

Suppose $V$ has dimension $n+1=2m$, the characteristic of $\mathbb{F}$ is not $2$, and $b$ is a symmetric bilinear form on $V$ satisfying 
$b(e_{i},e_{j})=0$ unless
$\{i,j\}=\{2k-1,2k\}$ for some $k$ with $1\leq k\leq m$, in which case 
$b(e_{i},e_{j})=1$.
Denote by $b_i$ the linear form $b_{e_i}$ as in \ref{exampleorth}.
Then the Siegel elements $x_{\epsilon_i-\epsilon_j}=e_{2i-1}\otimes b_{2j-1}-e_{2j}\otimes b_{2i}$, $x_{\epsilon_i+\epsilon_j}=e_{2i-1}\otimes b_{2j}-e_{2j-1}\otimes b_{2i}$
and $x_{-\epsilon_i-\epsilon_j}=e_{2i}\otimes b_{2j-1}-e_{2j}\otimes b_{2i-1}$
together with the elements $h_\alpha=[x_\alpha,x_{-\alpha}]$ form a Chevalley basis
of type $\mathrm{D}_m=\{\epsilon_i\pm \epsilon_j\mid 1\leq i, j\leq m, i\neq j\}$ in $\mathfrak{so}(V,b)$.
\end{exam}

A {\em long (short) root element}  in $\ch_\Phi(\MF)$  is an element in the $\mathrm{Aut}(\ch_\Phi(\mathbb{F}))$-orbit of an element $x_\alpha$ where $\alpha$ is a long (short) root in $\Phi$. 

We notice that, since $x_\alpha$ is $\mathrm{ad}$-nilpotent for both $\alpha$ being short or long, we find $\mathrm{exp}(x_\alpha,\lambda)$, where $\lambda\in \MF$,  to be an automorphism of $\ch_\Phi(\MF)$. Here $\mathrm{exp}(x_\alpha,\lambda)=\mathrm{exp}(\lambda x_\alpha)$ in the language of \cite{Car72}.

\begin{prop}
The long root elements in $\ch_\Phi(\mathbb{\MF})$ are extremal.
\end{prop}

\begin{proof}
Let $\alpha$ be a long root in $\Phi$.
As the characteristic of $\mathbb{F}$ is different from $2$, it suffices to
prove that $[x_{\alpha},[x_{\alpha},y]]\in \mathbb{F}x_{\alpha}$ for all $y$  in the Chevalley basis.

If $y=x_\beta$ for any $\beta\in \Phi$ different from $-\alpha$, 
then $[x_\alpha,[x_\alpha,y]]=[x_\alpha,[x_\alpha,x_\beta]]=0$ since at least one of $\alpha+\beta$ and
 $2\alpha+\beta$ is not in $\Phi$.

If $y=x_{-\alpha}$, then  $[x_\alpha,[x_\alpha,y]]=[x_\alpha,[x_\alpha,x_{-\alpha}]]=
[x_\alpha,h_\alpha]=-2 x_\alpha$.

Finally for  $y=h_\beta$ we have
$[x_\alpha,[x_\alpha,y]]=[x_\alpha,-2\frac{( \alpha,\beta)}{( \beta,\beta)} x_\alpha]=0$.
\end{proof}

\begin{lem}\label{generatedbylongroots}
The long root elements linearly span $\mathfrak{ch}_\Phi(\MF)$.
\end{lem}

\begin{proof}
First we show that every $x_\beta$, where $\beta$ is a short root in $\Phi$, is
a linear combination of long root elements.

If $\Phi$ is of type $\mathrm{B}_n$ or $\mathrm{C}_n$ and $\beta$ is a short root, then
we can find a long root $\alpha$ such that $2\beta+\alpha$ is long.
Indeed, if $\Phi$ is of type $\mathrm{B}_n$ then,
without loss of generality we can assume  $\beta=\epsilon_1$ and
take $\alpha=\epsilon_2-\epsilon_1$ and if $\Phi$ is of type $\mathrm{C}_n$
then $\beta$ can be assumed to be $-\epsilon_1+\epsilon_2$ and we can take $\alpha=2\epsilon_1$.
But then
$\exp(x_{-\alpha-\beta},1)(x_{\alpha+2\beta})=x_{\alpha+2\beta}-x_\beta-x_{-\alpha},$
so 
$x_\beta=-\exp(x_{-\alpha-\beta},1)(x_{\alpha+2\beta})+x_{\alpha+2\beta}+x_{-\alpha},$
a linear combination of three long root elements.
 
If $\Phi$ is of type $\mathrm{F}_4$, every short root is inside a $\mathrm{B}_2$ subsystem and we can apply the above.

So it remains to consider the $\mathrm{G}_2$-case and $\beta$ a short root.
Then we can pick a long root $\alpha$ such that
$\alpha, \alpha+3\beta$ and $2\alpha+3\beta$ are long.
Notice that $\alpha+\beta$ is short.
But then $\mathrm{exp}(x_{-\alpha-\beta},1)(x_{2\alpha+3\beta})+\mathrm{exp}(x_{-\alpha-\beta},1)(x_{2\alpha+3\beta})=2x_\beta+2x_{2\alpha+3\beta}.$

So $x_\beta=\frac{1}{2}(\mathrm{exp}(x_{-\alpha-\beta},1)(x_{2\alpha+3\beta})+\mathrm{exp}(x_{-\alpha-\beta},1)(x_{2\alpha+3\beta})-2x_{2\alpha+3\beta}),$
and hence a linear combination of three long root elements.

Now consider the elements $h_\alpha$, where $\alpha\in \Phi^+$.
We have

$$\mathrm{exp}(x_\alpha,1)(x_{-\alpha})=x_{-\alpha}+h_\alpha- x_\alpha,$$
from which we deduce that $h_\alpha$ is a linear combination of short or long root elements and hence also a linear combination of long root elements.
\end{proof}

Let $g$ be the extremal form on $\ch_\Phi(\MF)$, and denote by $\E$ the set of extremal points of $\ch_\Phi(\MF)$.

\begin{lem}\label{connectedness}
The extremal points $\langle x_\alpha\rangle$, with $\alpha\in \Phi$ a long root, are contained in a single connected component of the graph $(\E,\E_2)$.
\end{lem}

\begin{proof} 
We denote adjacency in $(\E,\E_2)$ by $\sim$.

Let $\alpha$, $\beta$ be distinct long roots and  $\langle x_{\alpha}\rangle$ and $\langle x_\beta\rangle$ be the corresponding extremal points.
Obviously, if $\alpha=-\beta$, the points are adjacent in the $(\E,\E_2)$.

If $\alpha\neq -\beta$ and   $( \alpha, \beta)\neq 0$,  
we can assume, (up to replacing $\beta$ by $-\beta$)  
that $\langle \alpha,\beta\rangle=-1$ and $\alpha+\beta$ is also a root. 
This implies that $[x_{\alpha}, x_{\beta}]=N_{\alpha,\beta}x_{\alpha+\beta}$ with $N_{\alpha,\beta}\neq 0$. By ~\ref{2generators}, it follows 
that $x_{\alpha}+x_{\alpha+\beta}$ is extremal. 

As $g(x_{-\alpha}, x_{\alpha}+x_{\alpha+\beta})=g(x_{-\alpha}, x_{\alpha})+g(x_{-\alpha}, x_{\alpha+\beta}) =1$, we find $\langle x_\alpha\rangle\sim\langle x_{-\alpha}\rangle\sim \langle x_\alpha+x_{\alpha+\beta}\rangle$.
Moreover, as  $g(x_{-\alpha-\beta},x_\alpha+x_{\alpha+\beta})=1$
we can also find a path from $\langle x_{\alpha+\beta}\rangle$ to $\langle x_\alpha+x_{\alpha+\beta}\rangle$ and hence from $\langle x_\alpha\rangle$ to $\langle x_{\alpha+\beta}\rangle$. Similarly, there is a path from $\langle x_\beta\rangle$ to $\langle x_{\alpha+\beta}\rangle$ and $\langle x_\alpha\rangle$ and $\langle x_\beta\rangle$ are in the same connected component. 

It remains to consider the case where $( \alpha,\beta)=0$, so the roots $\alpha$ and $\beta$ are orthogonal to each other.
If $\Phi$ is not of type $\mathrm{C}_n$, we can find a long root $\gamma$ such that
$(\alpha, \gamma)\neq 0$ and
$( \beta,\gamma)\neq 0$ and then apply the above, to conclude
that both $\langle x_\alpha\rangle$ and $\langle x_\beta\rangle$
are in the connected component of $(\E,\E_2)$ containing $\langle x_\gamma\rangle$.

If  $\Phi$ is of type $\mathrm{C}_n$, the long roots  are $\pm 2 \eps_i$ and the short roots are $\pm (\eps_i \pm \eps_j)$, where $1\mathopen<i\leq j\mathopen<n$. Note that in this case, we have the root lengths $\sqrt{2}$ and $2$.
We will show that for $i\neq j$ the elements $\langle x_{-2 \eps_i}\rangle$
and $\langle x_{2\eps_j}\rangle$ are in the same $(\E,\E_2)$-component.

Define $y=\exp(x_{-\eps_i-\eps_j},1)(x_{2\eps_i}) = x_{2\eps_i} - x_{\eps_i-\eps_j}- x_{-2\eps_j}$,  which is a long root element since $x_{2\eps_i}$ is one.

Bilinearity of $g$ gives
$
g(x_{-2\eps_i}, y)=g(x_{-2\eps_i}, x_{2\eps_i}) -g(x_{-2\eps_i}, x_{\eps_i-\eps_j})-g(x_{-2\eps_i}, x_{-2\eps_j})
=1+0+0\neq 0$.
Moreover
$
g(x_{2\eps_j}, y)=g(x_{2\eps_j}, x_{2\eps_i}) -g(x_{2\eps_j}, x_{\eps_i-\eps_j})-g(x_{2\eps_j}, x_{-2\eps_j})=0+0+1\neq 0$.
So we have a path   $\langle x_{-2\eps_i}\rangle\sim \langle y \rangle\sim \langle x_{2\eps_j}\rangle$ in $(\E,\E_2)$.
\end{proof}

As we have seen before, the radical of the extremal form $g$ is an ideal in
the Lie algebra $\ch_\Phi(\MF)$.
By $\overline{\ch_\Phi(\MF)}$ we denote the quotient Lie algebra
$\ch_\Phi(\MF)/{\mathrm{rad}(g)}$. For $x\in \ch_\Phi(\MF)$ we denote by $\bar x$ the element $x+\mathrm{rad}(g)\in \chf$.

The elements $\bar x$ in $\overline{\ch_\Phi(\MF)}$ where $x$ is a long root element not in $\mathrm{rad}(g)$, are also called long root elements.
The set of extremal points of $\chf$ is denoted by $\bar\E$.

\begin{prop}
The Lie algebra $\overline{\ch_\Phi(\MF)}$ is simple.
It is generated by its extremal elements.

These extremal elements are all scalar multiples of long root elements.
\end{prop}

\begin{proof}
Assume that $x\in \chf$ is extremal, but not in the same connected component of $(\E,\E_2)$ as some element $\bar x_{\alpha}\in \chf$ with $\alpha\in \Phi$ long. 
Since $\ch_\Phi(\MF)$ is generated by its long root elements, that are in one connected component by ~\ref{connectedness}, it follows by \ref{CIthm1} that $[x,\bar\chf]=0$ and $x\in Z(\chf)=\{0\}$. A contradiction. 
So $(\E,\E_2)$ and $\Gamma(\chf)$ are connected.

But this implies that all extremal points are in one single orbit under the automorphism group of $\chf$, see \ref{CIthm1}, and hence all extremal elements are scalar multiples of long root elements.
 
Now assume $\mathfrak{i}$ is a nonzero ideal of $\chf$.
Let $0\neq i\in \mathfrak{i}$. 
Then, as $\chf$ is spanned by the images of  long root elements,
there is an extremal $x$ with $\bar g(x,i)\neq 0$, where $\bar g$ is the extremal form on $\chf$.
But then $\bar g(x,i)x=[x,[x,i]]\in \mathfrak{i}$.

If $\langle y\rangle$ is in relation $\E_2$ with $\langle x\rangle$,
then inside $\langle x,y\rangle\cong \MSL$ we find that
$g(y,x)y=[y,[y,x]]$ is in $\mathfrak{i}$. By connectedness of $(\E,\E_2)$
we find all extremal elements to be in $\mathfrak{i}$ and hence $\mathfrak{i}=\chf$. This shows that $\chf$ is simple.
\end{proof}

If $\Phi$ is an irreducible root system of type $X_n$, $n\geq 2$, then
the extremal geometry of the Lie algebra $\chf$ is a root shadow space of
type $X_{n,J}$ for some set $J$, except when $X_n=\mathrm{A}_2$ and the characteristic of $\MF$ is $3$.
This is explained by the following.
The radical of $g$ is usually equal to the center of $\g$.
(See \cite{fle15} for details.)
The only exception is the case where $X_n=\mathrm{G}_2$ and
the characteristic of $\MF$ is equal to $3$.
In this exceptional case, the short roots are contained in the radical of
$g$. This radical is $7$-dimensional and $\overline{\mathfrak{ch}_\Phi(\MF)}$
is isomorphic to $\overline{\mathfrak{ch}_{\Phi'}(\MF)}$, where $\Phi'$ is a root system of type $\mathrm{A}_2$.  
 So, if  $X_n=\mathrm{A}_2$ and the characteristic of $\mathbb{F}$ is $3$, then the  extremal geometry is a root shadow space of type $\mathrm{G}_{2,\{1,2\}}$. (We notice that in characteristic $2$ more exceptions occur.)

The extremal geometry of $\chf$, when $\Phi$ is of type $\mathrm{A}_1$ or $\mathrm{C}_n$, does not contain lines.

\bigskip

Let $\g$ be a simple Lie algebra over a subfield $\mathbb{K}$ of $\mathbb{F}$
such that  
$\g\otimes_\mathbb{K}\mathbb{F}$ is isomorphic to $\hat\g$. Then we call
$\g$ a {\em form} of $\hat \g$.

Several forms of the Lie algebras $\chf$ do contain long root elements. 
Forms of Chevalley Lie algebras of type $\mathrm{A}$, $\mathrm{B}$, $\mathrm{C}$
and $\mathrm{D}$, give rise to various orthogonal and unitary Lie algebras generated by their long root elements.
The extremal geometries are then root shadow spaces of type $\mathrm{BC}_{n,2}$ or have no lines.
In case $\Phi$ is of type $\mathrm{D}_4$ one also obtains examples of extremal
geometries which are of type $\mathrm{G}_{2,2}$.
Forms of Chevalley Lie algebas of type $\mathrm{E}$ lead to
extremal geometries without lines, as well as 
root shadow spaces of type $\mathrm{A}_{2,\{1,2\}}$, $\mathrm{G}_{2,2}$, and $\mathrm{F}_{4,1}$. We obtain examples of Lie algebras generated by their extremal elements having an
extremal geometry isomorphic to the root shadow space of a building of algebraic type.  See for example \cite{Ti74,TW02,Wei03}.

\end{document}